\documentclass[11pt]{amsart}

\usepackage[english]{babel}
\usepackage[utf8x]{inputenc}
\usepackage[T1]{fontenc}

\usepackage[a4paper,top=3cm,bottom=3cm,left=3cm,right=3cm,marginparwidth=2cm]{geometry}

\usepackage{amssymb}
\usepackage{mathrsfs}
\usepackage{amsthm}
\usepackage{amsfonts}
\usepackage{amsmath}
\usepackage{graphicx}
\usepackage{url}

\usepackage[colorlinks=true, allcolors=blue]{hyperref}

\usepackage{lmodern}

\def\@rst #1 #2other{#1}
\renewcommand\MR[1]{\relax\ifhmode\unskip\spacefactor3000 \space\fi
  \MRhref{\expandafter\@rst #1 other}{#1}}
\renewcommand{\MRhref}[2]{\href{http://www.ams.org/mathscinet-getitem?mr=#1}{MR#1}}

\theoremstyle{plain}

\newtheorem{remarque}{Remark}[subsection]

\newtheorem{assumption}[remarque]{Assumption}

\newtheorem{theoreme}[remarque]{Theorem}

\newtheorem{definition}[remarque]{Definition}

\newtheorem{proposition}[remarque]{Proposition}

\newtheorem{lemma}[remarque]{Lemma}

\title{Sampling Random Cycle-Rooted Spanning Forests on Infinite Graphs}
\author{H. Constantin}

\newcommand\Prob{\mathbb{P}}

\renewcommand{\phi}{\varphi}
\renewcommand{\epsilon}{\varepsilon}

\begin{document}

\maketitle

\begin{abstract}
On a finite graph, there is a natural family of Boltzmann probability measures on cycle-rooted spanning forests, parametrized by weights on cycles. For a certain subclass of those weights, we construct Gibbs measures in infinite volume, as limits of probability measures on cycle-rooted spanning forests of increasing sequences of finite graphs. Those probability measures extend the family of already known random spanning forests and can be sampled by a random walks algorithm which generalizes Wilson's algorithm. We show that, unlike for uniform spanning forests, almost surely, all connected components are finite and two-points correlations decrease exponentially fast with the distance. 
\end{abstract}

\setcounter{tocdepth}{2}

\tableofcontents

 \section*{Introduction}
 
 We call \emph{cycle-rooted spanning forest}, on a finite connected graph, every subgraph which contains all vertices and all of whose connected components contain a unique cycle. Such a configuration of edges, endowed with a choice of orientation of cycles, can be seen as a discrete vector field on the graph: edges which are not in the cycle are oriented towards the cycle. Every vertex is associated with an edge starting from it. 
 
Given a connected finite graph~$G=(V,E)$, all of whose oriented cycles~$(\gamma)_{\gamma \in \mathcal C(G)}$ are endowed with positive weights~$(w(\gamma))$, we say that two oriented cycles are equivalent if they are equal after removal of their orientations. We define a probability measure on cycle-rooted spanning forests, induced by these cycle-weights as follows. Every spanning forest has a probability proportional to the product of weights of its cycles, counting both orientations:
\begin{equation} \label{defmeasure}
 \mu_w(F) = \frac{\prod_{[\gamma] \in \mathcal C(F)} (w(\gamma)+w(\gamma^{-1}))}{\mathcal Z_{w}},
 \end{equation}
where $\gamma, \gamma^{-1}$ are both oriented cycles of the equivalence class $[\gamma]$. The normalizing constant~$Z_{w}$ is called the \emph{partition function} of the model and is defined as follows:
$$ \mathcal Z_{w} = \sum_F \prod_{[\gamma] \in \mathcal C(F)} (w(\gamma)+w(\gamma^{-1})). $$

Recall that the usual model of uniform spanning tree is defined on finite graphs and its limit for infinite graphs is studied in~\cite{10.1214/aop/1176989121, benjamini_uniform_2001}. In these papers, the existence of a measure for infinite graphs is shown. This measure is sampled by an extension in infinite volume of Wilson's algorithm \cite{wilson} which does not depend on the ordering of vertices; see \cite{lyons_probability_2017} for a textbook treatment of this topic. These authors study the properties of the configurations under the measure in infinite volume, such as the number of connected components.

A model of random rooted spanning forests, all of whose connected components are rooted trees is studied in~\cite{boutillier_z-invariant_2017}. In this model, probability measures are associated with weights on vertices. If only one vertex has a weight, the probability measure associated with this weight has support in rooted trees, whose unique root is this vertex, and is equal to the uniform spanning tree measure after forgetting this root. In the model of rooted spanning forests, the random configurations are still sampled by an algorithm of loop-erased random walks. 

The model of rooted spanning forests is a particular case of the model of cycle-rooted spanning forests, which is studied in this article. Indeed, weights on vertices can be interpreted as weights on small self-loops over vertices and the roots of the random configuration can be seen as the unique cycles of their connected component. From this point of view, the model of cycle-rooted spanning forests is a generalization of the rooted spanning forests and of the uniform spanning tree.  

The measure on cycle-rooted spanning forests on finite graphs associated with a weight function on cycles for which every cycle has a weight smaller than 1 is studied in~\cite{10.1214/15-AOP1078}. This measure is sampled by a loop-erased random walks algorithm inspired from the Propp-Wilson algorithm for the generation of a random spanning tree. This algorithm does not depend on the ordering of vertices. A similar algorithm is also introduced in~\cite{PhysRevE.76.041140} to generate a random spanning web of square lattice annuli, using a ``cycle-popping'' inspired from the Propp-Wilson algorithm.

In this article, we study properties of probability measures on cycle-rooted spanning forests depending on a weight function on oriented cycles with values smaller than~1. We show that under an assumption on weights (Assumption~\ref{assumption}), the weak limit of measures on spanning forests on growing finite graphs, with cycle-weights smaller than~1, is well-defined and sampled by an algorithm of loop-erased random walks (Theorem~\ref{algothermo}) which does not depend on the ordering of vertices (Theorem~\ref{order}). We show furthermore that under this measure with a stronger assumption on weights (Assumption~\ref{assumption2}), all connected components are almost surely finite (Theorem~\ref{allccfinite}) and the decay of edge-edge correlations is exponential (Theorem~\ref{decayexpwilson}). Those properties show that under this assumption on weights, the measure which is constructed in infinite volume corresponds to a different ``phase'', in the sense of statistical mechanics, than the uniform spanning forests measure studied in~\cite{10.1214/aop/1176989121, benjamini_uniform_2001, lyons_probability_2017}. We also study probability measures on cycle-rooted spanning forests of finite graphs determined by a weight function~$w$ which can take values larger than 1. Those measures are no longer sampled by a loop-erased random walk algorithm. Nevertheless, we show that conditional on cycles of weights larger than~1, the measure is determined by a modified weight function~$w_-$ (Definition~\ref{def:w-}) which takes values smaller than~1 (Theorem \ref{algoconditional}). Assuming the existence of an infinite volume measure~$\mu_w$, we show that when Assumption~\ref{assumption2} is satisfied by this modified weight function~$w_-$, every connected component with a cycle is finite (Theorem~\ref{cyclefinite}). Combined with Proposition~\ref{onecycle} which says that almost surely every finite connected component has a cycle, this result implies that, almost surely, every connected component is either a finite cycle-rooted tree or an infinite tree. 

The paper is organized as follows. In Section~\ref{section1}, we define the probability measures on cycle-rooted spanning forests of finite graphs we are concerned with and give properties on these measures. In Section~\ref{section3}, we define the~$p$-loop erased random walks which generalize the loop-erased random walks and which are used in the third section to define a measure in infinite volume sampled by a random walk algorithm. In Section~\ref{section4}, we study the weak limit of probability measures on spanning forests on growing finite graphs, which gives the existence of a probability measure in infinite volume. In Section~\ref{section5}, we study the long-range behavior of the configurations under this probability measure, such as the non-existence of an infinite connected component and the exponential rate of decay of edge-to-edge correlations. In Section~\ref{section6}, we study properties of infinite configurations sampled by infinite volume probability measures which are determined by unbounded weights on cycles, provided the limit exists.

 \section{Measures on Cycle-Rooted Spanning Forests on finite graphs}
 
 \label{section1}

 \subsection{Cycle-rooted spanning forests}

Let~$G=(V,E)$ be a finite connected graph with vertex set~$V$ and edge set~$E$. For every subgraph~$F$ of~$G$, let~$\mathcal C(F)$ be the set of unoriented simple cycles of~$F$ and~$E(F)$ be the set of edges of~$F$. If~$[\gamma] \in \mathcal C(F)$ is a cycle of~$F$, denote by~$\gamma$ and~$\gamma^{-1}$ the two oriented cycles obtained from~$\gamma$. Let~$\mathcal C_\to (F)$ be the set of oriented cycles of~$F$.

We say that a subgraph of~$G$ is a \emph{cycle-rooted spanning forest} (CRSF) if it contains all the vertices and if each of its connected components contains a unique cycle. Let~$\mathcal U(G)$ be the set of CRSFs of~$G$.

Let~$w : \mathcal C_\to (G) \to \mathbb R_+$ be a non-zero function with non-negative values, defined on oriented cycles of~$G$. There is a natural probability measure on~$\mathcal U(G)$ associated with~$w$, which is denoted by~$\mu_{w}$. 
It is defined for every CRSF~$F \in \mathcal U(G)$ by:
\begin{equation} \label{defmeasures}
\mu_{w}(F) = \frac{  \prod_{[\gamma] \in \mathcal C(F)} (w(\gamma)+w(\gamma^{-1}))}{Z_{w}}, 
\end{equation} 
where $\gamma, \gamma^{-1}$ are both oriented cycles of the equivalence class $[\gamma]$, and~$Z_{w}$ is called the {partition function} of the model
\[
Z_{w} = \sum_{F \in  \mathcal U(G)} \prod_{[\gamma] \in \mathcal C(F)} (w(\gamma)+w(\gamma^{-1})). \]

We say that~$F$ is an \emph{oriented cycle-rooted spanning forest} (OCRSF) if it is a CRSF and every cycle of~$F$ is given an orientation, that is to say if every connected component contains a unique oriented cycle. Let~$\mathcal U_\to(G)$ be the set of OCRSFs of~$G$. Every edge of an OCRSF is oriented towards the cycle of its connected component. 

The partition function of the model can also be written as a sum of weights over OCRSF as follows:
\[
Z_{w} = 
 \sum_{F \ \text{OCRSF}}  \prod_{\gamma \in \mathcal C_\to(F)} w(\gamma), 
\]
and it gives a natural probability measure on OCRSF. 

Let~$W \subset V$ be a subset of vertices of~$G$. We say that~$F$ is a wired cycle-rooted spanning forest or essential cycle-rooted spanning forest (ECRSF) with respect to~$W$ if every connected component of~$F$ is either a unicycle disjoint from~$W$ or an unrooted tree which contains a unique vertex of~$W$, called a boundary-rooted tree. Let~$\mathcal U_W(G)$ the set of ECRSF with respect to~$W$.

\begin{definition}[Wired boundary conditions] \label{defwired}
We define a measure on~$\mathcal U_W(G)$ called the wired measure on ECRSF of~$G$ with boundary~$W$ whose configurations have weight proportional to the product of weights of cycles.
\begin{equation} \label{defwiredmeasures}
\mu^W_{w}(F) = \frac{ \prod_{[\gamma] \in \mathcal C(F)} (w(\gamma)+w(\gamma^{-1}))}{Z^W_{w}} 
\end{equation}
where $\gamma, \gamma^{-1}$ are both oriented cycles of the equivalence class $[\gamma]$.
\end{definition}

Notice that the measure defined in \eqref{defmeasures} corresponds to the case~$W= \emptyset$.

\subsection{Wilson's algorithm\label{algofini}}

When the weight function~$w$ is identically equal to 0, this measure~$\mu_w^W$ has support on ECRSF all of whose connected component are boundary-rooted trees. In particular, when~$W=\{r\}$ is a single vertex and the weight function~$w$ is identically equal to $0$, this measure has support on spanning trees rooted at~$r$. Since spanning trees on $G$ are in $1$-to-$1$ correspondence with spanning trees of $G$ rooted at $r$, this measure is independent of the choice of vertex~$r$ and gives to every spanning tree the same weight. Therefore, this measure is the uniform spanning tree measure defined for every tree~$T$ by:
\begin{equation} \label{tree}
\mu(T) = \frac{1}{Z_{tree}} 
\end{equation}
 where~$Z_{tree}$ is the partition function, that is the number of spanning trees of the graph~$G$. This measure~$\mu$ can be sampled by the Wilson algorithm (\cite{wilson}).

Assume that for every oriented cycle~$\gamma \in \mathcal C_\to (G)$, 
$$ w(\gamma) \in [0,1]$$
We will write~$p$ instead of~$w$ in the following when this assumption is satisfied.

According to~\cite{10.1214/15-AOP1078}, the measure~$\mu_{p}$ can be sampled by an algorithm of loop-erased random walk where we keep an oriented cycle~$\gamma$, with probability~$p(\gamma)$. 

More precisely, let~${x_1,\ldots,x_n}$ be an ordering of the vertex set~$V$ of~$G$ and let~$\mathsf F_0 = \emptyset$. At each step~$i$, let~$(X_n^{(x_i)})_{n \geq 0}$ be a simple random walk on the graph~$G$ starting from~$x_{i}$. Every time the random walk makes a loop, the oriented cycle~$\gamma$ is kept with probability~${p(\gamma)}$ or erased with probability~${1-p(\gamma)}$. The random walk~${(X_n^{(x_i)})_{n \geq 0}}$ is stopped when it reaches the set of already explored vertices denoted by~$V({\mathsf F_{i-1}})$ or when a cycle is kept. In the end of the~$i^{\text{th}}$ step, let~$\mathsf F_{i} = \mathsf F_{i-1} \cup L(X_n^{(x_i)})$ where~$L(X_n^{(x_i)})$ is obtained from~$(X_n^{(x_i)})_{n \geq 0}$ after removing all the loops except the last one if a loop is kept at the end of the~$i^{\text{th}}$ step. At the end,~$V(\mathsf F_n) = V(G_n)$. Notice that the algorithm always finishes if and only if there exists at least a loop~$\gamma$ in~$G$ such that~$p(\gamma)>0$.   

The measure~$\mu_{p}^W$ defined in Definition~\ref{defwired} can also be sampled by an algorithm. 
We follow the same algorithm but every time the random walk meets~$W$, the walk stops and a new random walk starts from the next vertex in the ordering. At the beginning of the algorithm, we set~$\mathsf F_0 = W$ instead of~$\mathsf F_0 = \emptyset$. The algorithm always finishes if and only if there exists at least a loop~$\gamma$ in~$G \backslash W$ such that~$p(\gamma)>0$ or~$W \neq \emptyset$.  

Let us emphasize that when~$W=\{r\}$ is a single vertex and the weight function~$w$ is equal to 0, then the sampling algorithm described just above is the classical Wilson algorithm which samples a uniform spanning tree on~$G$ rooted at~$r$.

\section{$p$-Loop erased random walks and rooting times} \label{lerw}

\label{section3}

In the following, we will consider a countably infinite connected graph~$G=(V,E)$, with finite degrees, exhausted by an increasing sequence~$(G_n)_{n \geq 1}$ of connected induced subgraphs of~$G$, with respective vertex set~$V_n$. We denote by~$\partial G_n$ the subset of~$V_n$ of vertices which are connected by an edge to the complement of~$G_n$ in~$G$.

For every~$v \in V$, we denote by~$\Prob_v$ the law of a simple random walk on~$G$ starting from~$v$.

We consider a weight function~$w=p \in [0,1]$ and we make the following assumption on the exhaustion~$(G_n)$ of the graph~$G$ and the weight function~$p$. 

\begin{assumption} \label{assumption}
There exists~$\alpha >0$ and~$\beta >0$, such that for every~$n \in \mathbb N^*$, for every random walk~$(X_n)_{n \geq 0}$ on~$G$, starting from a vertex~$v$ of~$\partial G_n$, there exists a loop~$\gamma_v$ in~$G_{n+1} \backslash (G_n \cup \partial G_{n+1})~$ which satisfies~$p(\gamma_v) \geq \alpha~$ and~$\Prob_v((X_1,\ldots, X_{\vert \gamma \vert}) = \gamma) > \beta$. 
\end{assumption}

\subsection{Hitting times, rooting time}

In the following we denote by~$v_0$ a vertex of~$G_1$.

\begin{definition}
If~$C$ is a subset of the vertex set~$V$, we define for a random walk~$(X_n)$ the hitting time of~$C$, that is to say
~$$ T_C := \min \{k \geq 0 \vert X_k \in C \}.$$
 Notice that in this definition,~$T_C$ can be equal to 0 if the random walk starts from a vertex of~$C$.
\end{definition}

\begin{definition}
Let~$(X_n)$ be a simple random walk starting from~$v_0$. 
Let~$(T_n)$ be the sequence of random hitting times of~$\partial G_n$ for the random walk~$(X_n)$, that is to say 
$$ T_n := T_{\partial G_n} = \min \{k \geq 0 \vert X_k \in \partial G_n \}.$$
\end{definition}

\begin{lemma}{\label{hitting}}
The hitting-time~$T_n$ is finite almost-surely for every~$n \in \mathbb N^*$. Furthermore,
$$ \lim_{k \rightarrow \infty} \Prob_{v_0}(T_n \geq k) = \Prob_{v_0}(T_n = \infty) = 0.$$
\end{lemma}

\begin{proof}
Let~$n \in \mathbb N^*$.
Almost surely,~$T_n$ is finite because almost surely if~$k \geq 1$, there exists a time such that the random walk makes~$k$ consecutive steps in the same direction. Therefore, the random walk exits every finite ball in finite time almost surely.
Since the events~$(T_n \geq k)$ are decreasing in~$k$ (for a fixed~$n$) with respect to inclusion, the monotone convergence theorem implies
$$ \lim_{k \rightarrow \infty} \Prob_{v_0}(T_n \geq k) = \Prob_{v_0}(\bigcap_{k \geq 1} \{T_n \geq k \}) =  \Prob_{v_0}(T_n = \infty),$$
which concludes the proof.
\end{proof}

Let~$(X_n)$ be a simple random walk on~$G$ starting from~$v_0$ and let~$(Y_n)$ be a sequence of independent random variables of uniform law on~$[0,1]$, which are independent of the random variables~$X_n$. 

We want to define a~$p$-loop erased random walk such that, if at time~$n$, the random walk~$(X_n)$ closes a loop~$\gamma_n$, the loop is kept if~$Y_n \leq p(\gamma_n)$ and erased else.  

Given~$(X_n, Y_n)$, we construct a sequence of random walks~$((Z_n^k)_n)_k$ as follows. 
We define recursively~$(Z_n^k)_{n \geq 1}$ for~$k \in \mathbb N^*$. Let~$(Z_n^1)=(X_n)$ and given~$(Z_n^k)_n$, let us consider the first time~$n_k$ such that~$Z_n^k$ closes a loop that is to say 
$$ n_k =  \min \{j > n_{k-1} \in \mathbb N^*  \vert Z^k_j \in \{Z^k_0,\ldots,Z^k_{j-1} \} \}.$$
Then, let~$n_k'$ be the time of the beginning of the loop, that is to say 
$$n_k' = \min\{ j \in \mathbb N, Z^k_j = Z^k_{n_k} \}.$$
 Therefore, the loop which is closed at time~$n_k$ is the loop~$\gamma_{n_k}:= (Z^k_{n_k'},\ldots ,Z^k_{n_k})$
Finally, if~$Y_{n_k} \geq p(\gamma_{n_k})$, then define for every~$n \in \mathbb N$,
$$ Z_n^{k+1} =
\begin{cases}
Z^k_{n_k} \quad \text{if} \ n_k' \leq n \leq n_k  \\
Z_n^k \quad \text{else}.  
\end{cases}
$$ 
$(Z_n^{k+1})$ is obtained from~$(Z_n^k)$ erasing the loop~$\gamma_{n_k}$.

Otherwise, if~$Y_{n_k} \leq p(\gamma_{n_k})$, for every~$m \geq k+1$ and~$n \in \mathbb N$, let
$$ Z_n^m =
\begin{cases}
Z^k_{n} \quad \text{if} \ n \leq n_k  \\
Z_{n_k}^k \quad \text{else}  
\end{cases}
$$ 
$(Z_n^m)$ is obtained from~$(Z_n^k)$ stopping the random walk at time~$n_k$.

\begin{definition}{\label{rooting time}}
If~$(X_n)$ is a simple random walk on~$G$ starting from~$v_0$ and~$(Y_n)$ is a sequence of independent random variables of uniform law on~$[0,1]$, which are independent of the~$X_n$, we say that~$(n_k)_{n \geq 1}$ is the sequence of random times where the random walk~$(X_n)$ closes a loop~$\gamma_{n_k}$. 
Let~$T_r$ be called the \emph{random rooting time} for~$(X_n, Y_n)$ that is to say the first time where a loop is kept:
$$ T_r := \min \{n_k \vert Y_{n_k} \leq p(\gamma_{n_k}) \},$$
where~$\min \emptyset = + \infty$.
If~$k$ is such that~$T_r = n_k$, then let~$(Z_n)_{n \leq T_r} = (Z_n^k)_{n \leq T_r}$ be called the~\emph{$p$-loop erased random walk} obtained from~$(X_n,Y_n)$. 
\end{definition}

Let us emphasize that if~$T_r$ is finite, then there exists a~$k$ such that~$T_r = n_k$ and then the~$p$-loop erased random walk~$(Z_n)_{n \leq T_r}$ is well defined and is obtained from~$(X_n)_{n \leq T_r}$, erasing every loop excepted the last one. Here, the loop-erased random walk is indexed on the same time set than the random walk~$(X_n)$. Nevertheless,~$Z_n$ does not depend only on~$(X_k)_{k \leq n}$.

 \subsection{The rooting time is almost surely finite.}

We will show in this subsection that the rooting time~$T_r$ is a stopping time and that almost surely, it is finite.

\begin{definition}
Let $(\mathcal{F}_n)_n$ be the filtration adapted to the process~$((X_n, Y_n))_n$  that is defined by
$$\mathcal F_n = \sigma(X_0,\ldots ,X_n, Y_0,\ldots ,Y_n),$$ which is the smallest sigma-field which makes the~$(X_i)_{0 \leq i \leq n}, (Y_i)_{1 \leq i \leq n}$ measurable. 
\end{definition}
 
 \begin{lemma}\label{stoppingtime}
 For every~$m \in \mathbb N^*$, the hitting time~$T_m$ is a stopping time with respect to the filtration~$(\mathcal F_n)_n$. The rooting time~$T_r$ is also a stopping time with respect to the filtration~$(\mathcal F_n)_n$. 
Moreover, for every~$n \in \mathbb N$, if we consider the~$\sigma$-field adapted to the stopping time~$T_n$, defined by 
$$ \mathcal F_{T_n} = \{A \in \mathcal F : \forall k \geq 0, \{T_n \leq k \} \cap A \in \mathcal F_k \},$$ 
then, the event~$\{T_n < T_r\}$ is in~$\mathcal F_{T_n}$. 
\end{lemma}

\begin{proof}
Let~$n \in \mathbb N$. The events~$\{T_n \geq k\} = \{X_1,\ldots , X_k \in G_n \backslash \partial G_n  \}$ are measurable with respect to~$\mathcal F_k$ and therefore~$T_n$ is a stopping time.  

From the construction of the rooting time, the event~$\{T_r \leq k\}$ only depends on the steps of the random walk before time~$k$, that is~$((X_n,Y_n))_{n \leq k}$ and therefore the event~$\{T_r \geq k\}$ is in~$\mathcal F_{k}$.

The event~$\{T_n < T_r\}$ is in~$\mathcal F_{T_n}$ because if~$k \geq 0$, 
$$ \{T_n \leq k \} \cap \{T_n < T_r\} = \bigcup_{1 \leq i \leq k} \left( \{T_n = i\} \cap \{T_r > i \}   \right) \in \mathcal F_k.$$
which concludes the proof.
\end{proof}

Let us emphasize that~$T_r$ is a stopping time for the filtration~$(\mathcal F_n)$ even if~$(Z_n)$ is not adapted to the filtration~$(\mathcal F_n)$. Indeed,~$Z_n$ depends on~$(X_k)$ for~$k \geq n$. Lemma~\ref{stoppingtime} is a useful tool to show that the rooting time is almost surely finite for a simple random walk starting from~$v_0$.

\begin{lemma}\label{rooting}
Under Assumption \ref{assumption}, the rooting time~$T_r$ for a simple random walk~$(X_n)$ starting from~$v_0$ and~$(Y_n)$ as defined in Definition~\ref{rooting time} is finite almost surely and the sequence~$(\Prob_{v_0}(T_r > T_n))_n$ decays exponentially fast to 0 with~$n$. More precisely, there exists~$\delta \in ]0,1[$ such that
$$  \Prob_{v_0}(T_{n} < T_r) \leq \delta^n.~$$
\end{lemma}

\begin{proof}
Let~$n \in \mathbb N^*$ be fixed. The process~$((X_k,Y_k))_{k \geq 0}$ satisfies the strong Markov property. Therefore, conditional on the event~$\{ T_n < \infty \}$ which is almost sure by Lemma~\ref{hitting}, for every~$k \geq 0$, the pair of random variables~$(X_{T_n +k}, Y_{T_n+k})$ is independent of~$\mathcal F_{T_n}$ given~$(X_{T_n}, Y_{T_n})$.

From Assumption~\ref{assumption}, there exists a loop~$\gamma_{X_{T_n}}$ which lies inside~$G_{n+1} \backslash (G_n \cup \partial G_{n+1})$ with weight larger than~$\alpha$ and such that the probability that a random walk~$(X_{T_n+k})_k$ makes this loop~$\gamma_{X_{T_n}}$ is greater than~$\beta$.

Let us denote by~$A_{X_{T_n}}$ the event that the random walk~$(X_{T_n+k})_k$ makes this loop~$\gamma_{X_{T_n}}$, and let us denote by~$B_{X_{T_n}}$ the event~$\{ Y_{T_n +\vert \gamma_{X_{T_n}} \vert} \leq p(\gamma_{X_{T_n}}) \}$. Conditional on~$X_{T_n}$, the events~$A_{X_{T_n}}$ and~$B_{X_{T_n}}$ are independent and have probabilities~$\Prob_{X_{T_n}}(\gamma_{X_{T_n}}) \geq \beta$ and~$p(\gamma_{X_{T_n}}) \geq \alpha$.  

The event~$A_{X_{T_n}} \cap \{ Y_{T_n + \vert \gamma_{X_{T_n}} \vert } \leq p(\gamma_{X_{T_n}}) \}$ for the random walk~$(X_{T_n+k}, Y_{T_n+k})_{k \geq 0}$ starting from~$(X_{T_n}, Y_{T_n} )\in \partial G_n$ has a probability greater than~$\alpha \beta$. Conditional on~$(X_{T_n}, Y_{T_n} )$, it is independent of~$\mathcal F_{T_n}$, therefore it is independent of the event~$\{ T_{n} < T_r \}$. 

Let us show that conditional on~$T_n < T_r$, if the event~$A_{X_{T_n}} \cap \{ Y_{T_n + \vert \gamma_{X_{T_n}} \vert } \leq p(\gamma_{X_{T_n}}) \}$ is satisfied, then the event~$\{T_{n+1} > T_r\}$ is satisfied. The idea is that on this event, the random walk keeps a loop before reaching~$\partial G_{n+1}$ and therefore~$T_{n+1} > T_r$.

Let~$i$ be the largest integer such that~$n_i \leq T_n$. Then, by construction of the~$p$-loop-erased random walk, assuming~$T_r > T_n \geq n_i$,~$(Z_k^{i+1})$ coincides with~$(X_k)$ after time~$n_i$ and therefore after time~$T_n$. For~$k \leq T_n$,~$Z_k^{i+1} \in \{X_0,\ldots ,X_{T_n}\}$ by construction and therefore~$Z_k^{i+1} \in G_n$. 

If the event~$A_{X_{T_n}} \cap \{ Y_{T_n + \vert \gamma_{X_{T_n}} \vert } \leq p(\gamma_{X_{T_n}}) \}$ is satisfied, then, for~$k$ between $T_n +1$ and~$T_n + \vert \gamma_{X_{T_n}} \vert$, we have~$Z_k^{i+1} \in G_{n+1} \backslash (G_n \cup \partial G_{n+1})$, and therefore, for such~$k$,
$$Z_k^{i+1} \notin (Z_0^{i+1},\ldots ,Z_{T_n}^{i+1}).$$ 

Since we have~$n_{i+1} \geq T_n$ by assumption on~$i$, we have necessarily~$n_{i+1}=T_n + \vert \gamma_{X_{T_n}} \vert$. Since the event~$\{ Y_{T_n + \vert \gamma_{X_{T_n}} \vert } \leq p(\gamma_{X_{T_n}}) \}$ is satisfied by assumption and~$T_r > n_i$, we have~$T_r = n_{i+1} = T_n + \vert \gamma_{X_{T_n}} \vert$, and since~$A_{X_{T_n}}$ is satisfied, for~$T_n +1 \leq k \leq T_n + \vert \gamma_{X_{T_n}} \vert$, we have~$X_k \in G_{n+1} \backslash (G_n \cup \partial G_{n+1})$ and therefore~$T_{n+1} > T_n + \vert \gamma_{X_{T_n}} \vert = T_r$.

Therefore, denoting by~$\delta := 1- \alpha \beta < 1$,
\[
\begin{split}
 \Prob_{v_0}(T_{n+1} < T_r \ \vert \  T_{n} < T_r) & \leq 1-\Prob_{(X_{T_n},Y_{T_n})}( A_{X_{T_n}} \cap \{Y_{T_n +\vert \gamma_{X_{T_n}} \vert } \leq \alpha \} \ \vert \  T_{n} < T_r)  \\
& = 1-\Prob_{(X_{T_n},Y_{T_n})}(A_{X_{T_n}} \cap \{Y_{T_n +\vert \gamma_{X_{T_n}} \vert } \leq \alpha \})  \\
& = 1-\Prob_{X_{T_n}}(A_{X_{T_n}})\Prob(Y_{T_n +\vert \gamma_{X_{T_n}} \vert } \leq \alpha )  \\
& \leq 1- \alpha \beta  = \delta.  
\end{split}
\]
This inequality holds for every~$n \in \mathbb N^*$ and~$\delta$ does not depend on~$n$.
Then, writing
$$ \Prob_{v_0}(T_{n+1} < T_r) = \Prob_{v_0}(T_{n+1} < T_r \ \vert \   T_{n} < T_r) \Prob_{v_0}(T_{n} < T_r),$$
we obtain by induction on~$n$ the exponential decay of the following probability :
$$ \Prob_{v_0}(T_{n} < T_r) \leq \delta^n.~$$

Let~$\varepsilon > 0$. For fixed~$n$ large enough,~$\delta^n \leq \varepsilon$. 

For~$k$ large enough, we have~$\Prob_{v_0}(T_n \geq k) \leq \varepsilon$. Then for~$k$ large enough, we have

\[
\begin{split}
 \Prob_{v_0}(T_r \geq k) & = \Prob_{v_0}(\{T_r \geq k\} \cap \{T_n \geq k\}) + \Prob_{v_0}(\{T_r \geq k\} \cap \{T_n \leq k-1\})  \\
 & \leq  \Prob_{v_0}(T_n \geq k) + \Prob_{v_0}(T_r > T_n) \leq 2\varepsilon 
 \end{split}
 \]
Therefore we have shown that for every~$\varepsilon >0$, for~$k$ large enough,~$ \Prob_{v_0}(T_r \geq k)  < 2 \varepsilon$, which means that 
$$\lim_{k \rightarrow \infty} \Prob_{v_0}(T_r \geq k) = 0 .$$
Therefore, from the monotone convergence theorem, 
$$ \Prob_{v_0}(T_r = \infty) = \Prob_{v_0}(\bigcap_{k \geq 1} \{T_r \geq k\}) = \lim_{k \rightarrow \infty} \Prob_{v_0}(T_r \geq k) = 0~$$ 

This concludes the proof.
\end{proof}

The proof of Lemma~\ref{rooting} can be adapted to show that the rooting time is almost surely finite for a random walk starting from another vertex of~$G$, even if this vertex is not in~$G_1$, as follows. 

\begin{lemma}{\label{translated}}
Let~$(X_n^{(x)})_n$ be a random walk starting from~$x \in G$ and and let $(Y_n)_n$ be the process defined in Definition~\ref{rooting time}. Under Assumption \ref{assumption}, the rooting time~$T_r$ for~$(X_n^{(x)})_n$ is finite almost surely and~$\Prob_x(T_r > T_n)$ decays exponentially fast to 0 with~$n$.
\end{lemma}

\begin{proof}
Notice that~$x$ is not anymore in~$G_1$ and therefore the bound~$\Prob_x(T_n \leq T_r ) \leq \delta^n$ does not hold. 

Nevertheless, if~$m_x$ is such that~$x \in G_{m_x}$, then for~$n \geq m_x$, the proof of Lemma~\ref{rooting} shows that for the loop-erased random walk starting from~$x$, 
$$ \Prob_x(T_{n+1} < T_r \vert T_{n} < T_r) \leq \delta~$$
and therefore for~$n \geq m_x$,
$$ \Prob_x(T_n < T_r) \leq \delta^{n-m_x}~$$
Therefore~$\Prob_x(T_n \leq T_r )$ tends to 0 exponentially fast with~$n$ and an argument similar to the one given in the proof of Lemma~\ref{rooting} gives that~$T_r$ is finite almost surely. 
\end{proof}

Lemma \ref{translated} shows that if we start a simple random walk on~$G$ from a vertex~$v$, almost surely~$T_r$ is finite. It implies that almost surely the sequence~$((Z_n^k)_{n \geq 0})_{k \geq 0}$ is constant eventually and its limit~$(Z_n)_{n \geq 0}$ is well defined with~$(Z_n)_{n \geq 0}$ constant for~$n \geq T_r$.  

\medskip

In the following, we define~$p$-loop-erased random walks with wired boundary conditions in order to adapt the usual Wilson algorithm to sample cycle-rooted spanning forests of infinite graphs.

\subsection{$p$-loop-erased random walk with a boundary condition}

Let us briefly recall our current notations. We still assume that~$(X_n)$ is a simple random walk on~$G$ starting from any vertex~$v$,~$(Y_n)$ is a sequence of independent random variables of uniform law in~$[0,1]$, which are independent of the~$X_n$ and~$W \subset V$ is a deterministic set of vertices.

We define in this subsection a~$p$-loop erased random walk obtained from~$(X_n,Y_n)_{n \geq 0}$ with the boundary condition~$W$.

\begin{definition}
Let~$T_W$ be the hitting time of~$W$, and let~$T_r$ be the rooting time of~$(X_n,Y_n)_{n \geq 0}$. Let~$T_f = \min(T_r, T_W)$ be called the \emph{ending time} of~$(X_n,Y_n)_{n \geq 0}$ with boundary condition~$W$.
\end{definition}

Given~$(X_n,Y_n)_{n \leq T_W}$, we construct a~$p$-loop erased random walk~$(Z^W_n)_n$ with boundary conditions~$W$ as follows. We define recursively~$n_k^W$ and~$(Z_n^{k,W})_{n \geq 0}$ for~$k \in \mathbb N^*$. Let~$(Z^{1,W}_n)_{n \leq T_W} = (X_n)_{n \leq T_W}$ and~$n_0^W=0$. 

Then, we define recursively a sequence~$((Z_n^{i,W})_{n \leq T_W})_{ i \geq 1}$ and a sequence~$(n_i)_{i \leq k}$ as follows. Let~$n_i^W : (X_0,\ldots ,X_k) \mapsto \min \{n_{i-1}^W < j \leq T_W \vert Z^{i,W}_j \in \{Z^{i,W}_0,\ldots , Z^{i,W}_{j-1} \} \}$ be the~$i$-th loop-closing time before reaching~$W$, where~$\min\emptyset=\infty$.
If~$n_i^W(X_0,\ldots ,X_k)=\infty$, let~$Z^{i+1,W} = Z^{i,W}$. 

Else, let~$n'^W_i$ be the first time~$j < n_i^W$ such that~$Z^{i,W}_{n_i^W}=Z^{i,W}_{n_i'^W}$ and if~$Y_{n_i^W}\leq p(\gamma_{n_i^W})$, let for every~$m \geq i+1$, 
$$Z_n^{m,W} = 
\begin{cases}
Z_n^{i,W} \quad \text{if} \ n \leq n_i^W,  \\
Z_{n_i^W}^{i,W} \quad \text{else},  
\end{cases}$$  
and otherwise, let for~$n \leq T_W$
$$Z_n^{i+1,W} = \begin{cases}
Z^{i,W}_{n_i^W} \quad \text{if} \ n_i'^W \leq n \leq n_i^W,  \\
Z_n^{i,W} \quad \text{else}.  
\end{cases}$$

Notice that~$Z_n^{i+1,W}$ is obtained from~$Z_n^{i,W}$ by erasing the first loop which ends before~$T_W$. While~$i$ is small enough such that~$n_i \leq T_W$, we have~$n_i^W = n_i$ and~$Z^{i+1,W}=Z^{i+1}$. 

\begin{proposition}\label{endingtime}
Almost surely,~$T_f$ is finite and~$((Z_n^{i,W})_{n \leq T_f})_{i \geq 0}$ is constant eventually. 
We define the~$p$-loop erased random walk ($p$-LERW) with boundary conditions~$W$ as 
$$(Z^W_n)_{n \leq T_f} = \lim_{i \rightarrow \infty} (Z_n^{i,W})_{n \leq T_f}.$$ 
\begin{itemize}
\item
If~$T_f = T_W$,~$(Z^W_n)_{n \leq T_f} = (Z^{i_f}_n)_{n \leq T_W}~$ where~$i_f = \min \{i \vert n_i > T_W\}$.  
\item
If~$T_f = T_r$,~$(Z^W_n)_{n \leq T_f} = (Z_n)_{n \leq T_f}$ where~$(Z_n)$ is the~$p$-loop erased random walk without any boundary condition. 
\end{itemize}
\end{proposition}

\begin{proof}
Recall from Lemma~\ref{translated} that~$T_r$ is finite almost surely. Since~$T_f \leq T_r$, the ending time~$T_f$ is almost surely finite. 
Assume that~$T_r < \infty$. 
Recall that the sequence~$(n_i)$ is strictly increasing. 
If~$T_W < T_r$, then~$T_W$ is finite and there exists~$i$ such that~$n_i > T_W$ and then~$n_{i}^W= \infty$ and~$Z^{m,W}=Z^{i,W}$ for~$m \geq i$.  
Let~$i_f = \min \{i \vert n_i > T_W\}$. Then~$n_{i_f-1} \leq T_W$. Then,~$(Z^{i_f,W}_n)_{n \leq T_W} = (Z^{i_f}_n)_{n \leq T_W}$. 

Then, for~$m \geq i_f$,~$n_m > T_W$ and then~$n_m^W = \infty$. Then, for every~$m \geq i_f$, 
$$(Z_n^{m,W})_{n \leq T_f} = (Z^{i_f, W}_n)_{n \leq T_f} = (Z^{i_f}_n)_{n \leq T_f}.$$

Else, there exists~$i$ such that~$T_r = n_i \leq T_W$. Then,~$n_i^W=n_i$ and~$Y_{n_i^W} \leq p(\gamma_{n_i^W})$ and for~$m \geq i$,~$(Z_n^{m,W})_{n \leq T_r} = (Z_n^{i,W})_{n \leq T_r} = (Z_n^{i})_{n \leq T_r}~$.
Since~$n_i = T_r$, then we have~$(Z_n^{i})_{n \leq T_r} =(Z_n)_{n \leq T_r}$ and therefore,~$(Z_n^{W})_{n \leq T_f} = (Z_n^{W})_{n \leq T_f}$ where~$(Z_n)$ is the~$p$-loop erased random walk without any boundary condition. 
\end{proof}

Notice that a~$p$-loop-erased random walk with boundary condition~$W$ is obtained from~$(X_n,Y_n)_{n \leq \min(T_r, T_W)}$ erasing every loop except the last one if~$T_r < T_W$.

Let us emphasize that when~$T_r > T_W$, the~$p$-loop erased random walk with boundary conditions~$(Z_n^W)_{n \leq T_W}$ is not equal to the~$p$-loop erased random walk~$(Z_n)_{n \leq T_W}$ stopped at~$T_W$.

\bigskip

We will define, in the next section, sequences of probability measures on CRSF on a growing exhaustion~$(G_n)$ of a countably infinite connected graph~$G$, with boundary conditions. We will see that under some hypotheses, those sequences of probability measures on CRSF of finite graphs~$G_n$ converge to thermodynamic limits which are probability measures on CRSF of the infinite graph~$G$. We will also define probability measures on CRSF of infinite graphs from~$p$-loop-erased random walks and compare those probability measures with limits of sequences of probability measures on finite graphs.

\section{Measures on CRSF in infinite volume and thermodynamic limits}

In the following, let~$G = (V,E)$ be a countably infinite connected graph, with finite degrees, and let~$(G_n)$ be an exhaustion of~$G$, in the sense of an increasing sequence of finite subgraphs of~$G$ whose union is~$G$. Let~$v_0$ be a vertex of~$G_1$.

\label{section4}
\subsection{Topological facts and boundary conditions}
Every subgraph of~$G$ can be seen as an element of~$\{0,1\}^{E}$. Let us recall some topological facts about the space~$\{0,1\}^{E}$. 

Since~$\{0,1\}$ is compact,~$\Omega = \{0,1\}^{E}$ is compact for the product topology and this topology is compatible with the following metric 
$$ d(w,w') = \sum_{e\in E} 2^{-\lVert e_- \rVert}1_{ \{w_e \neq w_e'\}},$$
where~$\lVert e_- \rVert$ is the length of the shortest path between~$v_0$ and an extremity of the edge~$e$. Therefore~$\Omega$ is a compact metric space.

A function~$f : \Omega \to \mathbb R~$ is continuous for the product topology if for every~$\epsilon >0$, there exists a finite subset~$\Lambda \subset E$, such that
$$ \sup_{w,w' \in \Omega  : w_{\vert \Lambda} = w'_{\vert \Lambda} } \vert f(w)-f(w') \vert \leq \epsilon.$$

A function~$f : \Omega \to \mathbb R~$ is called local if there exists a finite set~$\Lambda \subset E$ such that~$f(w)$ is entirely determined by~$w_{\vert \Lambda}$. We say that an event~$A \subset \Omega$ has finite support if the function~$1_A$ is a local function. The set of local functions is dense in the set of continuous functions~$(\mathcal C(\Omega), \vert \vert . \vert \vert_\infty)$ which is a Banach-space. 

We consider~$\mathcal C$ the smallest~$\sigma$-field which makes the cylinders~$\mathcal C_{\Lambda, \eta}= \{w \in \Omega, w_{\Lambda}=\eta \}$ measurable for every finite subset~$\Lambda \subset E$ and every finite configuration~$\eta \in \{0,1\}^\Lambda$. We say that a sequence of probability measures~$(\mu_n)$ converges to the measure~$\mu$ on~$(\Omega, \mathcal C)$ if and only if 
$$ \lim_{n \rightarrow \infty} \int_\Omega f d\mu_n = \int_\Omega f d\mu,$$
for every local function~$f$.
Since the set of local functions is dense in the set of continuous functions, this topology on the set of measures on~$(\Omega, \mathcal C)$ is the weak convergence.

In this section, we are interested in the sequences of measures~$(\mu_n)_{n \geq 1}$ on CRSF on growing subgraphs~$G_n$. If such a sequence of measures converges weakly towards an infinite volume measure, we have the following result on the limit measure.

\begin{proposition}\label{onecycle}
Assume that a sequence~$(\mu_n)_{n \geq 1}$ of measures on CRSF on growing subgraphs~$G_n$ of~$G$ converges weakly towards a measure~$\mu$, and let~$F$ be distributed according to~$\mu$. Then~$\mu$-almost surely, every finite connected component of~$F$ has exactly one cycle and its cycle has non-trivial weight.
\end{proposition}

\begin{proof}
Let~$x \in G$ and let~$T$ be a finite connected subgraph of~$G$ which contains~$x$ and which satisfies one of the following properties:
\begin{itemize}
\item
T has strictly more than one cycle;
\item
T has a cycle of trivial weight;
\item
T has no cycle. 
\end{itemize}
Let~$cc(x)$ be the connected component of~$x$ in~$F$. Notice that the event~$\{ cc(x)=T \}$ is an event with finite support since its support is included in the set of edges which have at least one extremity in~$T$. 

Let~$m$ be large enough such that~$T \subset G_{m-1}$. Then, the event~$$\{ cc(x)=T \} = \{ cc(x) \cap G_m=T \}~$$ has support in~$G_m$. For every~$n \geq m$,~$\mu_n$ is supported on CRSF whose cycles have non trivial weight on~$G_n$. Therefore,
$$ \mu_n(cc(x) \cap G_m =T) = \mu_n(cc(x)_{F_n} =T)  =  0.$$

We have the convergence~$\mu_n \rightarrow \mu$ on configurations with finite support. Then
\[
\mu_n(cc(x) \cap G_m = T) \rightarrow \mu(cc(x) \cap G_m =T).
\]
Finally, we obtain~$\mu(cc(x)=T)=0$. Since~$G$ is countable, almost surely, every finite connected component of~$F$ has exactly one cycle and its cycle has non-trivial weight.
\end{proof}

We can consider sequences of measures~$(\mu_n)_{n \geq 1}$ on CRSF on growing subgraphs~$G_n$ of~$G$ with boundary conditions, such as free and wired boundary conditions. 

\begin{definition}[Free boundary conditions]
We define the free measure on CRSF of~$G_n$ as the measure on CRSF of~$G_n$ whose configurations have weight proportional to the product of weights of cycles. This measure is denoted by~$\mu_n^F$.  
\end{definition}

\begin{definition}[Wired boundary conditions]
We define the wired measure on CRSF on~$G_n$ as the measure on CRSF of graph~$G_n$ whose configurations are either trees connected to~$\partial G_n$ or unicycles and have weight proportional to the product of weights of cycles. 
This measure is denoted by~$\mu_n^W$.  
\end{definition}

\subsection{Sampling algorithm for a fixed ordering on an infinite graph} \label{algo}

In the following, we still consider a countably infinite connected graph~$G$, an exhaustion~$(G_n)$ and a weight function~$w=p \in [0,1]$ on cycles which satisfies Assumption~\ref{assumption}. 

We now construct a probability measure on CRSF of~$G$ for weights~$p$ which is sampled by an algorithm determined by a fixed ordering of the vertex set~$V$.

Let~$\phi$ be an ordering of the vertex set~$V$ of~$G$, in the sense of a bijection~$\phi : \mathbb N \to V$. Let~$(v_i)_{i \geq 1}$ be the sequence of vertices of~$G$ with ordering~$\phi$, that is~$(v_i)_i=(\phi(i))_i$.

We will construct a measure on CRSF of~$G$ by means of a family of~$p$-loop-erased random walks with boundary conditions which are defined recursively. This family will be obtained deterministically from a family of independent simple random walks following results of Section~\ref{lerw}.

We still denote by~$\mathcal C$ the smallest~$\sigma$-field which makes the cylinders~$C_{\Lambda, \eta}$ measurable.

\begin{definition} \label{dataalgo}
Let~$((X_n^{(x)})_{n\geq 1})_{x \in G}, ((Y_n^{(x)})_{n\geq 1})_{x \in G}$ be independent random variables such that for all~$x \in G$,~$(X_n^{(x)})_n$ is a simple random walk on~$G$ starting from~$x$ and~$(Y_n^{(x)})_n$ is a sequence of independent random variables with uniform law on~$[0,1]$.
For a fixed~$x$, consider the sequence~$\left( (X_n^{(x)}, Y_n^{(x)}) \right)_{n \geq 1}$ and denote by~$T_r^{x}$ the rooting time of the~$p$-LERW that is to say the first time~$n$ such that~$(X_n^{(x)})_n$ closes a loop~$\gamma_n$ such that the inequality~$p(\gamma_n) \geq Y_n^{(x)}$ holds. 
\end{definition}

\begin{definition}{\label{construction}}
For a fixed random data~$((X_n^{(x)})_{n\geq 1})_{x \in G}, ((Y_n^{(x)})_{n\geq 1})_{x \in G}$ as above, we construct the subgraphs~$(\mathsf F_i)$ recursively. Let~$\mathsf F_0=\emptyset$. Let~$i \in \mathbb N^*$ and assume that~$\mathsf F_{i-1}$ is constructed. Denote by~$T_{f}^{\underline v_i}$ the ending time of~$((X_n^{(v_i)})_{n\geq 1}, (Y_n^{(v_i)})_{n\geq 1})$ with boundary condition~$V(\mathsf F_{i-1})$, that is
$$ T_f^{\underline v_i} = \min(T_r^{v_i}, T_{V(\mathsf F_{i-1})}).$$
Let~$\mathsf F_i = \mathsf F_{i-1} \cup L(v_i)$ where~$L(v_i)$ is the~$p$-LERW with boundary condition~$V(\mathsf F_{i-1})$ obtained from~$\left( (X_n^{(v_i)}, Y_n^{(v_i)}) \right)_{n\geq 1}$ until~$T_f^{\underline v_i}$.
\end{definition}

Each step~$i$ of the algorithm finishes either if the random walk reaches a connected component created during a previous step or if the random walk is rooted to a loop. Notice that~$T_{f}^{\underline v_i}$ is the time where the~$i^{th}$-step of the algorithm with ordering~$\phi$ finishes. Recall that under Assumption~\ref{assumption} on~$p$, Proposition~\ref{endingtime} implies that~$T_f^{\underline v_i}~$ is finite almost surely.

\begin{lemma}
There exists a measure~$\mu_\phi$ on~$(\mathcal U(G), \mathcal C)$ which is sampled by the previous algorithm with ordering~$\phi$. The measure on finite cylinders corresponds to finite random configurations which are sampled in a finite time. 
\end{lemma}

\begin{proof}
Sample a sequence~$((X_n^{(x)})_{n\geq 1})_{x \in G}, ((Y_n^{(x)})_{n\geq 1})_{x \in G}$ such as in Definition~\ref{dataalgo}. 

From Definition~\ref{construction}, we obtain a CRSF of~$G$ by setting~$\mathsf F = \cup_{i \geq 1} \mathsf F_i$. The configuration~$\mathsf F$ is well defined since it is a deterministic function of 
$$((X_n^{(x)})_{n\geq 1})_{x \in G}, ((Y_n^{(x)})_{n\geq 1})_{x \in G}.$$ 

Let~$\mu_\phi$ be the law of~$\mathsf F$ associated with a random choice of~$((X_n^{(x)})_{n\geq 1})_{x \in G}, ((Y_n^{(x)})_{n\geq 1})_{x \in G}$, that is to say the push-forward by the algorithm of the measure which gives 
$$((X_n^{(x)})_{n\geq 1})_{x \in G}, ((Y_n^{(x)})_{n\geq 1})_{x \in G}.$$  

Then, $\mu_\phi$ is a probability measure on~$(\mathcal U(G), \mathcal C)$. Let~$B$ be a finite subset of size~$n$ of~$E$, with edges~$e_1,\ldots ,e_n$ and let~$\varepsilon_1,\ldots ,\varepsilon_n \in \{0,1\}^n$. Let~$K$ be the finite set of vertices containing all the extremities of edges of~$B$ and vertices which are preceding those vertices for the order~$\phi$. Let us consider the previous algorithm for vertices~$v_1,\ldots ,v_{\vert K \vert }$ for the ordering~$\phi$. Almost surely, the algorithm to construct~$\mathsf F_{\vert K \vert}$ finishes in a finite time. The constructed graph~$\mathsf F_{\vert K \vert}$ is a random subgraph of~$G$ which is spanning for~$K$ and therefore it is spanning for~$B$. 
Then~$\mu_\phi(C_{\varepsilon_1,\ldots ,\varepsilon_n})$ is the probability that the random configuration~$F_{\vert B}$ obtained from the previous construction satisfies 
$$\mathsf F_{\vert B} \in C_{\varepsilon_1,\ldots ,\varepsilon_n}.$$ 
The measure~$\mu_\phi$ restricted to~$(2^B, \mathcal C)$ is the law of the random configuration~$\mathsf F_{\vert B}$, which is sampled in a finite time. 
\end{proof}

In the following, we will show that the measure in infinite volume constructed from an enumeration of the vertex set~$V$ does not depend on the choice of the enumeration. The proof of this statement will rely on a comparison between the measure~$\mu_\phi$ for an ordering~$\phi$ and a measure on CRSF on a large finite subgraph of~$G$. We will see that, under some hypotheses, the thermodynamic limit coincides with the measure sampled by the previous algorithm and does not depend on the ordering of the infinite vertex set.

\subsection{Thermodynamic limits of the Wilson measures}

Assume that Assumption~\ref{assumption} on the existence of a lower bound~$\alpha >0$ on the weight of a family of loops still holds.  

Let~$n \in \mathbb N$ and~$\phi_n$ be an ordering of the vertex set~$V_n$ of the graph~$G_n$. The measures~$\mu_n^F$ and~$\mu_n^W$ are sampled by the algorithm described in the first section. 

Let us consider the sequence of measures~$(\mu_n)$ on CRSF of~$(G_n)_{n\geq 1}$ which are defined by the previous algorithm but if the walk meets~$\partial G_n$, the walk is stopped. According to~\cite{10.1214/15-AOP1078}, the measure does not depend on the ordering of the vertices of~$G_n \backslash \partial G_n$.

\begin{theoreme}{\label{algothermo}}
Let~$\phi$ be an ordering of~$G$ in the sense of a bijection~$\phi : \mathbb N \to G$. Let~$(G_n)$ be an increasing exhaustion of~$G$, and let~$(\mu_n^{F}), (\mu_n^{W})$ be the corresponding sequences of probability measures on CRSF of~$G_n$ with free and wired boundary conditions, respectively. The sequences of probability measures~$(\mu_n^{F})$ and~$(\mu_n^{W})$ converges weakly to the measure~$\mu_\phi$.
\end{theoreme}

 \begin{proof}
We consider an event~$B \in 2^E$ which depends only on finitely many edges, and we consider~$K_0$ the set of vertices incident to the edges on which~$B$ depends. Let~$K$ be the union of~$K_0$ and the set of vertices that precede some vertex in~$K_0$ in the ordering~$\phi$ of the vertices. Let~$n$ be large enough such that~$K \subset G_n$.

Let us construct a coupling~$(\mathsf F, \tilde{\mathsf F}_n^F,  \tilde{\mathsf F}_n^W)$ of random configurations, obtained from the same random data~$((X_n^{(x)})_{n\geq 1})_{x \in G}, ((Y_n^{(x)})_{n\geq 1})_{x \in G}$, such that the law of~$\mathsf F$ is~$\mu_\phi$, the law of~$\tilde{\mathsf F}_n^F$ is~$\mu_n^F$ and the law of~$\tilde{\mathsf F}_n^W$ is~$\mu_n^W$ and such that the three configurations coincide with high probability on~$B$. 

We denote~$\tilde \phi_n$ the ordering induced by the ordering~$\phi$ on~$G_n$. 

We follow the algorithm for every vertex of~$K$ following the ordering~$\phi_{\vert K}$. If at one step of the algorithm, the random walk~$(X_n^{(x)})$ starting from a vertex~$x \in G_n$ reaches~$\partial G_n$, the configuration~$\mathsf F$ is obtained following the random walk in the infinite graph~$G$ until the end of the step and~$\tilde{\mathsf F}_n^F$,~$\tilde{\mathsf F}_n^W$ are obtained following the random walk with boundary conditions. More precisely,~$\tilde{\mathsf F}_n^W$ is obtained from~$p$-loop erased random walks with boundary conditions~$\partial G_n$ and~$\tilde{\mathsf F}_n^F$ is obtained following the random walk on the graph~$G_n$, until the end of the step, that is the ending time of the process~$(X_n^{(x)}, Y_n^{(x)})_n$.

Once every vertex of~$K$ has been explored, we complete the configuration~$\mathsf F$ following the ordering~$\phi$ and we complete the configurations~$\tilde{\mathsf F}_n^F$,~$\tilde{\mathsf F}_n^W$ on~$G_n$ following the ordering~$\tilde \phi_n$, with boundary conditions. 

Let us denote by~$E(K)$ the set of edges whose vertices are in~$K$. The configurations~$\mathsf F$,~$\tilde{\mathsf F}_n^F$ and~$\tilde{\mathsf F}_n^W$ obtained from the algorithm are respectively subgraphs of~$G$ and~$G_n$. The three configurations are spanning subgraphs of~$G_n$ and in particular spanning subgraphs of~$(K,E(K))$.  

$\mathsf F_{E(K)}$ is the restriction to~$(K,E(K))$ of a random configuration following the law~$\mu_\phi$. Since~$G_n$ is finite,~$(\tilde{\mathsf F}_n)^{W,F}_{E(K)}$ is the restriction to~$(K,E(K))$ of a random configuration following the law~$\mu_n^{W,F}$ and this law does not depend on~$\tilde \phi_n$ (see~\cite{10.1214/15-AOP1078}).  

Since~$B$ depends only on edges whose endpoints are in~$K_0$, we know that
$$\vert \mu_\phi(B) - \mu_n^{F,W}(B)  \vert \leq \Prob(\mathsf F_{E(K)} \neq (\tilde{\mathsf F}_n)_{E(K)})$$

Following the previous algorithm, while each step starting from a vertex of~$K$ finishes before the random walk reaches~$\partial G_n$, both configurations~$\mathsf F_{E(K)}, (\tilde{\mathsf F}_n)_{E(K)}$ which are obtained are equal. Therefore, from the union bound,

\[
\begin{split}
 \Prob(\mathsf F_{E(K)} \neq (\tilde{\mathsf F}_n)_{E(K)}) & \leq \Prob \left( \bigcup_{i \in [\vert K\vert ]} \{ T_n^{v_i} \leq T_f^{\underline v_i} \} \right)  \leq \sum_{i \in [\vert K\vert ]} \Prob(T_n^{v_i} \leq T_f^{\underline v_i} )    \\
 & \leq \sum_{i \in [\vert K\vert ]} \Prob(T_n^{v_i} \leq T_r^{v_i} ) \leq \vert K \vert \max_{i \in [\vert K\vert ]} \Prob(T_n^{v_i} \leq T_r^{v_i} ) 
 \end{split}
\]

From Lemma~\ref{translated}, we obtain when~$n \to \infty$,
$$ \vert \mu_\phi(B) - \mu_n^{W,F}(B)  \vert \leq  \vert K \vert \max_{i \in [\vert K\vert ]} \Prob(T_n^{v_i} \leq T_r^{v_i} )  \rightarrow 0$$
which implies the weak convergence of~$(\mu_n^{W,F})$ towards~$\mu_\phi$.
\end{proof}

Recall that for every~$n$, the measure~$\mu_n^W$ is sampled by an algorithm and does not depend on the ordering of vertices of~$G_n$ chosen in the algorithm. Combined with Theorem~\ref{algothermo}, this independence implies the following result.

\begin{theoreme}{\label{order}}
Let~$\phi$ be an ordering of the vertices, that is a bijection~$\phi:\mathbb N \to V$. Let~$p$ be a weight function satisfying Assumption~\ref{assumption}. Let~$\mu_\phi$ be the measure on the cycle-rooted spanning forests of~$G$ associated with the algorithm of loop-erased random walk with weights~$p(\gamma)$. The measure~$\mu_\phi$ does not depend on~$\phi$. 
\end{theoreme}

\begin{proof}
Let~$\phi$,~$\tau$ be two orderings, with~$(v_i)=(\phi(i))$,~$(w_i)=(\tau(i))$ let~$K_1$ and~$K_2$ be respectively the sets of vertices that precede some vertex in~$K_0$ in the ordering~$\phi$ (resp.~$\tau$) and let~$n$ large enough such that~$K_1 \cup K_2 \subset G_n$. Then, from Theorem~\ref{algothermo},

$$ \vert \mu_\phi(B) - \mu_\tau(B)  \vert \leq  \vert K_1 \vert \max_{i \in [\vert K_1 \vert ]} \Prob_{v_i}(T_n \leq T_r) + \vert K_2 \vert \max_{i \in [\vert K_2 \vert ]} \Prob_{w_i}(T_n \leq T_r) \rightarrow 0$$

This shows that both distributions in infinite volume coincide on cylinders and therefore the measure in infinite volume does not depend on the ordering of the vertices.
\end{proof}

For a weight function~$p$ satisfying Assumption~\ref{assumption}, we will denote by~$\mu_p$ the corresponding probability measure on CRSFs of G, which does not depend on the ordering of the vertices. \\

\section{Study of the configurations sampled under the Wilson measure}

\label{section5}

In this section, we will study the asymptotic behavior of configurations and the rate of decay of correlations with the distance for the measure~$\mu_p$ which is sampled by the previous algorithm of~$p$-loop erased random walks in infinite volume and which will be called the Wilson measure, under the following assumption. 

\begin{assumption}\label{assumption2}
There exists~$\alpha >0, \beta >0, M, M'>0, C>0, d \in \mathbb N$, a family of loops~$\Gamma \subset \mathcal C(G)$ and for every~$x \in G$, an increasing sequence~$(B_n^x)$ of subgraphs of~$G$, exhausting~$G$ and containing~$x$ such that : 
\begin{itemize}
\item
For every~$\gamma \in \Gamma$,~$\alpha \leq w(\gamma) \leq 1$.
\item
For every~$v \in \partial B_n^x$, there exists a loop~$\gamma_v \in \Gamma \cap (B_{n+1}^x \backslash (B_n \cup \partial B_{n+1}^x))$ such that the probability for a random walk starting from~$v$ of making this loop~$\gamma_v$ is greater than~$\beta$. 
\item
For every~$x$, for every~$n \in \mathbb N$,~$M'n \leq d(x, \partial B_n^x) \leq Mn$, and~$\vert \partial B_n^x \vert \leq Cn^d$.
\end{itemize}
\end{assumption}

Assumption~\ref{assumption2} implies Assumption~\ref{assumption} and is satisfied in particular if the graph and the weight function~$w$ on cycles are invariant under translations and if Assumption~\ref{assumption} is satisfied for an exhaustion~$(G_n)_n$ such that~$d(0,\partial G_n)\sim M n$ when~$n$ tends to infinity, where~$M>0$.

\subsection{Every connected component is finite for the Wilson measure} \label{finiteccwilsonmeasure}

\begin{definition}
For a vertex~$x$ and a subset~$A \subset G$, we denote by~$\{x \leftrightarrow A\}$ the event~$\{A \cap C_x \neq \emptyset \}$ where~$C_x$ is the connected component of~$x$ in the random configuration sampled under~$\mu$. In particular,~$\{x \leftrightarrow y\}$ means that~$x$ and~$y$ are in the same connected component.
\end{definition}

 We will denote by~$T_{m,x}^x$ the hitting-time of~$\partial B_m^x$ for the random walk~$(X_n^{(x)})$. Recall that~$T_m^x$ is the hitting-time of~$\partial G_m$ for the random walk~$(X_n^{(x)})$.

\begin{lemma}\label{upperbound}
Under Assumption~\ref{assumption2}, there exists~$\delta>0$ such that the following inequality holds for every~$m$, for every~$x \in G$
$$\Prob_x(\{T_r^x \geq T_{m,x}^x\})  \leq \delta^{m}.$$
\end{lemma}

\begin{proof}
Let~$x \in G$. Under Assumption~$\ref{assumption2}$, Assumption~\ref{assumption} is satisfied for the vertex~$x$, and therefore, if we denote by~$T_{m,x}^x$ the hitting time of~$\partial B_m^x$ for a~$p$-loop erased random walk starting from~$x$, and~$T_r^x$ its rooting time, Lemma~\ref{rooting} gives the existence of a~$0< \delta < 1$ such that the following inequality holds for every~$m$, 
$$\Prob_x(\{T_r^x \geq T_{m,x}^x\})  \leq \delta^{m},$$
where~$\delta = 1-\alpha \beta$ for parameters~$\alpha, \beta$ of~\ref{assumption2}. In particular,~$\delta$ does not depend on~$x$, which concludes the proof.
\end{proof}

\begin{lemma}\label{connectdistance}
Let~$\delta>0$ as in Lemma~\ref{upperbound} and~$M>0$ as in Assumption~\ref{assumption2}. Let~$x,y$ be two vertices of~$G$ and denote by~$d(x,y)$ the distance between~$x$ and~$y$, that is to say the length of the shortest path from~$x$ to~$y$. Then, if~$n \leq \frac{d(x,y)}{2M}$, 
$$ \mu_p(x \leftrightarrow y) \leq  2\delta^{n}.$$
\end{lemma}

\begin{proof}
According to Theorem~\ref{order}, the measure~$\mu_p$ does not depend on the ordering of the vertices. We may choose an ordering~$\phi$ in which~$x$ and~$y$ are the first two vertices. 
Since~$n \leq \frac{d(x,y)}{2M}$, we have~$d(x,\partial B_n^x) \leq Mn \leq \frac{d(x,y)}{2}$ and~$d(y,\partial B_n^y) \leq Mn \leq \frac{d(x,y)}{2}$.

If~$x$ and~$y$ are in the same connected component in a configuration obtained from this algorithm, we know that either for the~$p$-loop erased random walk starting from~$x$ or for the one starting from~$y$, we have~$\{T_r^x \geq T_{f \phi}^x \geq T_{n,x}^x\}$ or~$\{T_r^y \geq T_{f \phi}^y \geq T_{n,x}^y\}$. Indeed, if~$T_{f \phi}^x \leq T_{n,x}^x$ and~$T_{f \phi}^y \leq T_{n,y}^x$, then the~$p$-loop erased random walk starting from~$x$ and from~$y$ cannot intersect, and form two disjoint connected component in the configuration.  
Therefore, from the union bound,
$$ \mu_p(x \leftrightarrow y) \leq \Prob_x(\{T_r^x \geq T_{n,x}^x\}) + \Prob_y(\{T_r^y \geq T_{n,y}^y\}) \leq 2\delta^{n} ~$$
which concludes the proof.
\end{proof}

\begin{theoreme}{\label{allccfinite}}
$\mu_p$-almost surely, for every vertex~$x \in V$ of~$G$, the connected component of~$x$ is finite.
\end{theoreme}

\begin{proof}
Let~$x \in G$. For every~$n \in \mathbb N$, for every~$y \in \partial B_n^x$,~$d(x,y) \geq M'n$. Let~$n' = \lfloor \frac{M'n}{2M} \rfloor$. Then~$n' \leq \frac{d(x,y)}{2M}$ and therefore, from Lemma~\ref{connectdistance}, 
$$ \mu_p(x \leftrightarrow y) \leq  2 \delta^{n'}.$$

Then, from the union bound, the following upper bound on the probability that the connected component of x contains vertices of the boundary of~$B_n^x$ holds for every~$n \in \mathbb N$, with~$\delta' = \delta^{\frac{M'}{2M}}$:

$$ \Prob(x \leftrightarrow \partial B_n^x) \leq \sum_{y \in \partial B_n^x} \Prob(x \sim y) \leq 2 \vert \partial B_n^x \vert  \delta^{\lfloor \frac{M'n}{2M} \rfloor} \leq 2C n^d \delta'^n.$$

Then, from the monotone convergence theorem, we have
$$ \Prob(x \leftrightarrow \infty) = \Prob(\cap_n \{x \leftrightarrow \partial B_n^x\} ) = \lim_{n \to \infty} \Prob(x \leftrightarrow \partial B_n^x) = 0.$$
Since~$G$ is countable, we know that~$\mu$-almost surely, for every~$x \in G$,
the connected component of~$x$ is finite.
\end{proof}

 \subsection{Exponential decay of correlations for the Wilson measure}

We still assume that weights are in~$[0,1]$ and satisfy Assumption~\ref{assumption2}.

Let~$m \in \mathbb N$ and let~$e_1=(x_1,y_1)$ and~$e_2=(x_2,y_2)$ be such that 
$d(\{x_1,y_1\},\{x_2,y_2\}) \geq m$.

If~$\mathsf F$ is a CRSF following the law~$\mu_p$ it can be sampled from the algorithm described in Section~\ref{algo} and from Theorem~\ref{order}, it does not depend on the chosen ordering of vertices, therefore we may assume that the first four vertices of the ordering~$\phi$ are~$x_1, y_1,x_2,y_2$. 

Let us consider in the following, four independent couples of sequences of random variables~$(X^{x_1}_n, Y^{x_1}_n)$,~$(X^{y_1}_n, Y^{y_1}_n)$,~$(X^{x_2}_n, Y^{x_2}_n)$,~$(X^{y_2}_n, Y^{y_2}_n)$, 
as defined in Section~\ref{finiteccwilsonmeasure}.

For~$i \in \{1,2\}$, let us denote by~$A_i$ the event that both~$p$-loop erased random walks obtained from~$(X^{x_i}_n, Y^{x_i}_n)_n$,~$(X^{y_i}_n, Y^{y_i}_n)_n$ starting from~$x_i,y_i$ are rooted before leaving the subgraphs~$B_{m/2}^{x_i}, B_{m/2}^{y_i}$, that is to say that
$$ A_i = \{ T_r^{x_i} < T_{x_i,m/2}^{x_i} \} \cap \{ T_r^{y_i} < T_{y_i,m/2}^{y_i} \} ~$$

\begin{lemma}{\label{indep}}
Conditional on~$A_1 \cap A_2$, the events~$ \{e_1 \in \mathsf F\}$ and~$\{e_2 \in \mathsf F \}$ are independent. 
\end{lemma}

\begin{proof}

Let~$\mathsf F_4$ be the subgraph obtained after the first four runs of the algorithm. 

Notice that, once~$\mathsf F_4$ has been sampled, during every subsequent run of the algorithm, the~$p$-loop erased random walk stops if it reaches~$x_1,y_1,x_2,y_2$ because~$\mathsf F_4$ contains~$x_1,y_1,x_2,y_2$.
Therefore, for~$\mathsf F$ the configuration obtained following the algorithm in infinite volume, we have the following equality of events for~$i \in [1,2]$,
$$  \{e_i \in \mathsf F\} = \{ e_i \in \mathsf F_4 \}.~$$

Let~$(Z_n^{x_1})_{n \leq T_r^{x_1}}$ be the~$p$-loop erased random walk obtained from~$(X^{x_1}_n, Y^{x_1}_n)$ and let 
$$W_1 = V((Z_n^{x_1})_{n \leq T_r^{x_1}})$$
 be the set of vertices explored by this~$p$-loop erased random walk.
Let~$(Z_n^{y_1})_{n \leq \min(T_r^{y_1},T^{y_1}_{W_1})}$ be the~$p$-loop erased random walk starting from~$y_1$ with boundary condition~$W_1$.
Let~$\mathfrak F_1$ be the subgraph given by~$(Z_n^{x_1})_{n \leq T_r^{x_1}}, (Z_n^{y_1})_{n \leq \min(T_r^{y_1},T^{y_1}_{W_1})}$.

Let~$(Z_n^{x_2})_{n \leq T_r^{x_2}}$ be the~$p$-loop erased random walk obtained from~$(X^{x_2}_n, Y^{x_2}_n)$ and let 
$$ W_2 = V((Z_n^{x_2})_{n \leq T_r^{x_2}})$$
 be the set of vertices explored by this~$p$-loop erased random walk.
Let~$(Z_n^{y_2})_{n \leq \min(T_r^{y_2},T^{y_2}_{W_2})}$ be the~$p$-loop erased random walk starting from~$y_1$ with boundary condition~$W_2$.
Let~$\mathfrak F_2$ be the subgraph given by~$(Z_n^{x_2})_{n \leq T_r^{x_2}}, (Z_n^{y_2})_{n \leq \min(T_r^{y_2},T^{y_2}_{W_2})}$.

Let us emphasize that the~$p$-loop erased random walk corresponding to the third and the fourth runs of the algorithm has boundary conditions~$V(\mathfrak F_1)$, corresponding to the configuration created during the first two runs of the algorithm. Therefore, in general,~$\mathfrak F_1$ and~$\mathfrak F_2$ are not disjoint and their union is not the component created after four runs of the algorithm.

If~$A_1$ is satisfied,~$\mathfrak F_1$ is contained in~$B_{m/2}^{x_1} \cup B_{m/2}^{y_1}$ and if~$A_2$ is satisfied, 
$$\begin{cases}
T_r^{x_2} < T_{x_2,m/2}^{x_2} < T_{V(\mathfrak F_1)} \\
T_r^{y_2} < T_{y_2,m/2}^{y_2} < T_{V(\mathfrak F_1)}
\end{cases}$$
 and therefore, the third and the fourth runs finish before the~$p$-loop erased random walks reach~$V(\mathfrak F_1)$, that is to say that the~$p$-loop erased random walks with boundary condition~$V(\mathfrak F_1)$ coincides with the~$p$-loop erased random walks without this boundary condition (see Proposition~\ref{endingtime}).
Therefore, if~$A_1$ and~$A_2$ are satisfied,~$\mathfrak F_1$ and~$\mathfrak F_2$ are disjoint connected components and their union is exactly the component created after four runs of the algorithm.

In particular, if~$A_1 \cap A_2$ is satisfied, for~$i \in [1,2]$,~$\{e_i \in \mathsf F\}$ is satisfied if and only if~$\{e_i \in \mathfrak F_i\}$ is satisfied.

We show that conditional on~$A_1 \cap A_2$, the random configurations~$\mathfrak F_1$ and~$\mathfrak F_2$ are independent. Recall that~$(Z_n^{x_1}), (Z_n^{y_1})$ and~$(Z_n^{x_2}), (Z_n^{y_2})$ are independent and~$\mathfrak F_1, A_1$ only depends on~$(Z_n^{x_1}), (Z_n^{y_1})$ and~$\mathfrak F_2, A_2$ only depends on~$(Z_n^{x_2}), (Z_n^{y_2})$. Therefore, if~$F_1, F_2$ are some fixed configurations, 
$$ \Prob(\mathfrak F_1 = F_1, A_1, \mathfrak F_2=F_2,A_2) = \Prob(\mathfrak F_1 = F_1, A_1)\Prob(\mathfrak F_2=F_2,A_2).$$

Therefore, using independence of~$\{\mathfrak F_i = F_i\} \cap A_i$ and~$A_j$ for~$i \neq j$, we obtain
\[
\begin{split}
\Prob(\mathfrak F_1 = F_1,\mathfrak F_2=F_2 \vert A_1, A_2) & =
\frac{\Prob(\mathfrak F_1 = F_1, A_1)\Prob(\mathfrak F_2=F_2,A_2) }{\Prob(A_1 \cap A_2)}  \\
&= \Prob(\mathfrak F_1=F_1 \vert A_1)\Prob(\mathfrak F_2=F_2 \vert A_2)  \\
& = \Prob(\mathfrak F_1 = F_1 \vert A_1 \cap A_2)\Prob(\mathfrak F_2=F_2 \vert A_1 \cap A_2).   
\end{split}
\]
Therefore, conditional on~$A_1 \cap A_2$, the random variables~$\mathfrak F_1$ and~$\mathfrak F_2$ are still independent.

Therefore, conditional on~$A_1 \cap A_2$,~${\{e_1 \in \mathsf F\}}=\{e_1 \in \mathfrak F_1\}$ and~${ \{e_2 \in \mathsf F\}}=\{e_2 \in \mathfrak F_2\}$ are independent. 
\end{proof}

As a consequence, we obtain the following decay of correlations. 

\begin{theoreme}\label{decayexpwilson}
There exists a parameter~$0 < \iota<1$ such that for every~$m$ large enough,
$$ \mu_p(e_2 \in \mathsf F)\mu_p(e_1 \in \mathsf F) - \iota^m  \leq  \mu_p(\{e_2 \in \mathsf F\} \cap \{e_1 \in \mathsf F\}) \leq \mu_p(e_2 \in \mathsf F)\mu_p(e_1 \in \mathsf F) + \iota^m~$$ 
\end{theoreme}

\begin{proof}
Let us compute~$\mu_p(\{e_1, e_2 \in \mathsf F\})~$ using the following decomposition :
$$ \{e_1, e_2 \in \mathsf F\} = \left( \{e_1, e_2 \in \mathsf F\} \cap A_1 \cap A_2 \right) \cup \left( \{e_1, e_2 \in \mathsf F\} \cap (A_1^\complement \cup A_2^\complement) \right).$$

Since~$\{e_1,e_2 \in \mathsf F\} \cap  (A_1^\complement \cup A_2^\complement)$ is included in~$A_1^\complement \cup A_2^\complement$, it has probability less than the quantity~$\mu_p(A_1^\complement \cup A_2^\complement)~$.  

From Lemma~\ref{rooting}, there exists some~$\delta <1$ such that from the union bound, we get
$$ \mu_p(A_1^\complement \cup A_2^\complement) \leq 4 \delta^{m/2}.~$$

For the other term, we use the independence of~$\{e_1 \in \mathsf F\}$ and~$\{e_2 \in \mathsf F\}~$ conditional on~$A_1 \cap A_2$ proved in Lemma~\ref{indep}, which implies
\[ \begin{split}
 \mu_p(\{e_2 \in \mathsf F\} \cap \{e_1 \in \mathsf F\} \cap A_1 \cap A_2) 
 &= \mu_p(\{e_1 \in \mathsf F\} \cap \{e_2 \in \mathsf F\} \vert A_1 \cap A_2) \mu_p(A_1 \cap A_2)  \\
 & = \frac{\mu_p(\{e_1 \in \mathsf F\} \cap A_1 \cap A_2) \mu_p(\{e_2 \in \mathsf F\} \cap A_1 \cap A_2)}{\mu_p(A_1 \cap A_2)}     \\
 & \leq \frac{\mu_p(\{e_1 \in \mathsf F\}) \mu_p(\{e_2 \in \mathsf F\})}{\mu_p(A_1 \cap A_2)}.      
 \end{split}
 \]
 
 Using again the lower bound on~$\mu_p(A_1 \cap A_2)$ which comes from Lemma~\ref{rooting}, we have
$$ \mu_p(A_1 \cap A_2) \geq 1-4  \delta^{m/2}.$$

Let~$\eta < 1~$ be such that for~$m$ large enough,
$$ 4\delta^{m/2} < \eta^m~$$

Therefore, we have the following upper bound on~$ \frac{1}{\mu_p(A_1 \cap A_2)}$,

$$ \frac{1}{\mu_p(A_1 \cap A_2)} \leq \frac{1}{1-\eta^m} = \sum_{k \geq 0} \eta^{mk} = 1 + \sum_{k \geq 1} \eta^{mk} \leq 1+ \sum_{k \geq m} \eta^{k} \leq  1 + \frac{ \eta^m}{1-\eta}.$$

Therefore we get
$$ \mu_p(\{e_2 \in \mathsf F\} \cap \{e_1 \in \mathsf F\} \cap A_1 \cap A_2)  \leq  \mu_p(e_1 \in \mathsf F)   \mu_p(e_2 \in \mathsf F) (1 +   \frac{\eta^m}{1-\eta}   ).$$

For the other inequality, notice that
\[ \begin{split}
& \frac{ \mu_p(\{e_1 \in \mathsf F\} \cap A_1 \cap A_2) \mu_p(\{e_2 \in \mathsf F\} \cap A_1 \cap A_2)  }{\mu_p(A_1 \cap A_2)}   \\
 & \geq   (\mu_p(\{e_1 \in \mathsf F\} \mu_p(\{e_2 \in \mathsf F\}) - 2 \mu_p(A_1^\complement \cup A_2^\complement)  \\
 & \geq (\mu_p(\{e_1 \in \mathsf F\} \mu_p(\{e_2 \in \mathsf F\}) - 2\eta^m. 
  \end{split}
 \]
Therefore,
$$ \mu_p(\{e_2 \in \mathsf F\} \cap \{e_1 \in \mathsf F\} \cap A_1 \cap A_2) \geq (\mu_p(\{e_1 \in \mathsf F\} \mu_p(\{e_2 \in \mathsf F\}) - 2\eta^m,$$
and
$$ \mu_p(\{e_2 \in \mathsf F\} \cap \{e_1 \in \mathsf F\}) \geq (\mu_p(\{e_1 \in \mathsf F\} \mu_p(\{e_2 \in \mathsf F\}) - 2\eta^m.$$
Considering~$\iota<1$ such that for~$m$ large enough,~$2 \eta^m  < \iota^m~$ and~$\eta^m \frac{1}{1-\eta}  + 4 \delta^{m/2}   < \iota^m~$ concludes the proof.
\end{proof}

\section{Study of the configurations sampled under an infinite volume measure~$\mu_w$.}

\label{section6}

In this section, we consider a non-negative weight function~$w$ on oriented cycles of~$G$, which can take values larger than~1. We assume that the sequence of measures~$(\mu_n)$ on cycle-rooted spanning forests of~$G_n$ associated with the weight function~$w$ converges weakly towards an infinite volume measure~$\mu$ and that this measure does not depend on the free or wired boundary conditions.

\subsection{Algorithm conditional on cycles with weights larger than 1.}

In this subsection, we assume that~$G$ is a finite connected graph. Let~$W \subset G$ be a set of vertices of~$G$ (which can be empty). Let~$\mathcal C_+(G \backslash W)$ be the set of cycles of~$G$ of weight strictly larger than~1, so-called positive cycles and~$\mathcal C_-(G \backslash W)$ the set of cycles of~$G$ of weight less than~1, so-called negative cycles.

\begin{definition}\label{def:w-}
The weight function~$w_-$ on cycles of the graph~$G$ associated with the weight function~$w$ is defined by the following restrictions
$$ \begin{cases}{w_-}_{\mathcal C_-(G \backslash W)}= w_{\mathcal C_-(G \backslash W)}, \\
{w_-}_{\mathcal C_+(G \backslash W)}=0. \\
\end{cases} $$
\end{definition} 

The following result gives a way to sample a cycle-rooted spanning forest conditional on its positive cycles.

\begin{theoreme}{\label{algoconditional}}
Let~$C$ be a subset of~$\mathcal C_+(G \backslash W)$ and let~$A$ be the set of vertices which are extremities of edges in~$C$. Let~$\mathsf F$ a ECSRF with respect to~$W$ sampled according~$\mu^W$. Conditional on~$\mathcal C_+(\mathsf F) = C$,~$\mathsf F \backslash C$ has the same law than a ECRSF with respect to~$A \cup W$ with weight function~$w_-$.
\end{theoreme}

\begin{proof}
Let~$F_0 \in \mathcal U_W(G)$.
$$ \mu^W(\mathsf F=F_0 \vert \mathcal C_+(F) = C) = \frac{\mu(\mathsf F=F_0 \cap \mathcal C_+(\mathsf F) = C) }{\mu(C_+(\mathsf F) = C) }.$$
Notice that this quantity is zero if~$C_+(F_0) \neq C$.
Then, if~$C_+(F_0) = C$ and~$F_0 \in \mathcal U_W(G)$, every connected component of~$F_0$ either contains a unique cycle in~$C_-(G)$ or a unique cycle in~$C$ or is connected to a unique point in~$W$. Therefore, every connected component of~$F_0 \backslash C$ either contains a unique cycle in~$C_-(G)$ or or is connected to a unique point in~$W \cup A$ which means that~$F_0 \backslash C \in \mathcal U_{W \cup A}(G)$.

Then, the measure~$\mu^W(. \vert C_+(F) = C)$ has support in 
$$\mathcal U^C_W(G) = \{ F \in \mathcal U_{W}(G) \vert C_+(F) = C \} = \{ F \in \mathcal U_{W}(G) \vert F = C \cup F_-, F_- \in \mathcal U_{W \cup A}(G) \},$$ 
and if~$F_0 \in \mathcal U^C_W(G)$,
\[
\begin{split}
 \mu^W(\mathsf F=F_0 \vert \mathcal C_+(F) = C) & = \frac{ \prod_{\gamma \in C } w(\gamma) \prod_{\gamma \in \mathcal C_-(F_0)} w(\gamma) }{ \sum_{F \in \mathcal U_W(G) \vert \mathcal C_+(F) = C } \prod_{\gamma \in C} w(\gamma) \prod_{\gamma \in \mathcal C_-(F)} w(\gamma) }  \\
& = \frac{\prod_{\gamma \in \mathcal C_-(F_0)} w(\gamma) }{ \sum_{F \in \mathcal U^C_W(G)}\prod_{\gamma \in \mathcal C_-(F)} w(\gamma) }. 
 \end{split}
 \]

Writing every~$F \in \mathcal U^C_W(G)$ on a unique way as~$C \cup F_-$ with~$F_- \in \mathcal U_{W \cup A}(G)$,
\[
\begin{split}
 \mu^W(\mathsf F=F_0 \vert \mathcal C_+(\mathsf F) = C) & =  \frac{\prod_{\gamma \in \mathcal C_-(F_{0-})} w(\gamma) }{ \sum_{F_- \in \mathcal U_{W \cup A}(G)} \prod_{\gamma \in \mathcal C_-(F_-)} w(\gamma) }  \\
& = \mu^{W \cup A}_{w_{\mathcal C_-(G\backslash W)} } (F_{0-}) = \mu^{W \cup A}_{w_{\mathcal C_-(G\backslash W)} } (F_{0}\backslash C). 
\end{split}
\]

Finally,
$$  \mu^W(\mathsf F\backslash C = . \vert \mathcal C_+(\mathsf F) = C) = \mu^{W \cup A}_{w_{\mathcal C_-(G\backslash W)} } (.),$$ 
which concludes the proof.
\end{proof}

Since~$w_{\mathcal C_-(G\backslash W)}$ takes values in~$[0,1]$ by definition of~$C_-(G\backslash W)$, the measure~$\mu^{W \cup A}_{w_{\mathcal C_-(G\backslash W)} }~$ can be sampled by the wired Wilson algorithm with boundary conditions~$A \cup W$.

Therefore, under the measure~$\mu^W$, conditional on~$\mathcal C_+(F)$, a ECRSF with respect to~$W$ has the same law as a a ECRSF with respect to~$W$ and extremities of edges in~$\mathcal C_+(F)$ and can be sampled from a loop-erased random walk algorithm.

\subsection{All connected components with a cycle are finite}

We assume that Assumption~\ref{assumption2} holds for the weight function~$w_-$ as defined in Definition~\ref{def:w-} and for a family of cycles~$\Gamma \subset \mathcal C_-(G)$. In particular, Assumption~\ref{assumption} also holds for~$w_-$, but the weight function~$w$ can take values larger than 1.

We will show in this subsection that under this assumption, every connected component with a cycle is finite. 

 We will use the following lemma on the exponential decay of the tail distribution of ending times, which is a corollary of Section~\ref{section3}. 
 
 \begin{lemma}\label{rootinghitting}
Let~$m \in \mathbb N$,~$n \geq m$. Let~$C \subset G_n$ be a subgraph of~$G_n$.
Let~$\Prob_x$ be the law of a~$w_-$-loop erased random walk~$(X_n)$ starting from x and~$T_r$ be the rooting time of~$(X_n)$. Let~$T_C$ and~$T_{m,x}$ be the hitting times of~$C$ and~$\partial B_m^x$. Under Assumption~\ref{assumption} on~$w_-$, the following inequality holds
~$$ \Prob_x(\min(T_{C}, T_r)  \geq T_{m,x}) \leq \delta^m.$$
 \end{lemma}
 
 \begin{proof}
 Applying Lemma~\ref{rooting} to the random walk~$(X_n)$ with weight function~$w_-$ which satisfies Assumption~\ref{assumption}, we get 
~$$ \Prob_x(T_r  \geq T_{m,x}) \leq \delta^m.$$
But since~$\min(T_{C}, T_r) \leq T_r$, we have
$$ \Prob_x(\min(T_{C}, T_r)  \geq T_m^x) \leq \Prob_x(T_r  \geq T_{m,x}) \leq \delta^m,$$
which concludes the proof.
\end{proof}

  Let~$x \in V$ be a fixed vertex of~$G$. We introduce some events with compact support which depend on~$x$ and prove an upper bound on the probability of those events.
  
 \begin{definition}\label{defprobaconnect-}
For~$m \in \mathbb N$, we denote by~$A_{m} = \{x \leftrightarrow \partial B_{2m}^x \}$ the event that x and~$\partial B_{2m}^x$ are connected in~$F$, that is to say that there exists a path between~$x$ and~$\partial B_{2m}^x$ in~$B_{2m}^x$.
 If~$n \in \mathbb N$, we denote by~$ \Gamma^-_n = \{ x \leftrightarrow C_-(F_{G_n}) \}~$ the event that x is connected to a closed cycle in~$G_n$ of weight less than 1.  
 \end{definition}

 \begin{lemma}\label{probaconnect-}
Let~$m \in \mathbb N$. For~$n \geq m$ large enough, if~$F_n$ is distributed according to the free measure~$\mu_n$ on~$G_n$,
$$ \mu_n(A_m \cap \Gamma^-_n ) \leq (\vert \partial B_m^x \vert+1) \delta^m.$$
\end{lemma}

\begin{proof}
Let~$n$ be a large enough integer such that  for every~$y \in \partial B_{2m}^x$, we have~$B_m^y \subset G_n$. 
Let~$C \subset \mathcal C_+(G_n)$. From Theorem~\ref{algoconditional}, conditional on~$\mathcal C_+(F_n)=C$,~$F_n$ is given by an algorithm of~$w_-$-loop erased random walks with boundary conditions on~$C$. The proof relies on the same ideas as that in the proof of Lemma~\ref{connectdistance} and Theorem~\ref{allccfinite}.

The event~$A_m \cap \Gamma^-_n~$ is satisfied if there exists~$y \in \{x \} \cup \partial B_{2m}^x$ such that the~$w_-$-loop erased random walk starting from~$y$ has left~$B_m^y$ before being rooted to a cycle in~$\mathcal C_-(G_n)$ and before touching~$C$. 

From Lemma~\ref{rootinghitting}, for every~$y \in \{x \} \cup \partial B_{2m}^x$, 

$$ \Prob_y(\min(T_{C}, T_r)  \geq T_{m,y}) \leq \delta^m.$$
 Then, the union bound concludes the proof.
 \end{proof}
 
 \begin{lemma}\label{cycle-}
 The previous lemma implies that
~$$ \mu(\{x \leftrightarrow C_-(F) \} \cap \{ \vert cc(x) \vert = \infty \}) =0.$$
 \end{lemma}
 
 \begin{proof}
Let~$\epsilon >0$. Let~$m \in \mathbb N$ fixed, large enough such that~$ (\vert \partial B_m^x \vert+1)\delta^m<\epsilon$. We consider the notations from Definition~\ref{defprobaconnect-}.
 
Since~$(\Gamma^-_n )$ is increasing, if we let~$ \Gamma^- := \{x \leftrightarrow \mathcal C_-(F) \} = \cup \Gamma^-_n~$, then
$$\mu(A_m \cap \Gamma^-) = \mu(A_m \cap \cup \Gamma^-_n ) = \mu(\cup_n (A_m \cap \Gamma^-_n)) = \lim_n \mu(A_m \cap \Gamma^-_n).$$

Let us consider~$n_0$ large enough such that the inequality from Lemma~\ref{probaconnect-} holds. Since for~$n \geq n_0$, ~$ \mu_n(A_m \cap \Gamma^-_{n_0}) \leq \mu_n(A_m \cap \Gamma^-_n) \leq \epsilon,~$

 We obtain
~$$ \mu(A_m \cap \Gamma^-_{n_0}) = \lim_n \mu_n(A_m \cap \Gamma^-_{n_0}) \leq \epsilon.~$$

 It holds for every~$n_0$ large enough and therefore,~$ \mu(A_m \cap \Gamma^-) \leq \epsilon$.
 
 Since this inequality holds for~$m$ large enough, we have when~$m \rightarrow \infty$,~$ \mu(A_m \cap \Gamma^-) \rightarrow 0~$.

 Therefore, since~$(A_m)$ is decreasing and~$\cap_m A_m =  \{ \vert cc(x) \vert = \infty \}$, the monotone convergence theorem concludes the proof. 
 \end{proof}

 \begin{definition}\label{defprobaconnect+}
For~$l \in \mathbb N$, let~$\Gamma^+_l = \{x \leftrightarrow_{B_l^x} C_+(F_{B_l^x})\}$ be the event that x is connected inside~$B_l^x$ to a cycle with weight larger than 1 which is inside~$B_l^x$, that is to say the event that in~$F_{B_l^x}$, the connected component of x contains a cycle with weight larger than 1. For~$m \in \mathbb N$, let~$A_m := \{x \leftrightarrow \partial B_{m}^x \}$
be the event that x is connected to the boundary of~$B_m^x$. 
 \end{definition}

 \begin{lemma}
Let~$m_0 \in \mathbb N$. For~$m$ large enough, there exists~$n_0$ such that if~$n \geq n_0$ and if~$F_n$ is distributed according to~$\mu_n$, then,
~$$ \mu_n(A_m \cap \Gamma^+_l) \leq \vert \partial B_{m_0}^x \vert  \delta^{m_0}.$$
 \end{lemma}

 \begin{proof}\label{bordcycle+}
Let~$m_0, m \in \mathbb N$. Assume that~$m$ is large enough such that for every~$y \in \partial B_m^x$, the equality~$B_{m_0}^y \cap B_l^x = \emptyset$ holds.
Let~$n$ be large enough such that for every~$y \in \partial B_m^x$, the inclusion~$B_{m_0}^y \subset G_n$ holds.
  
Let~$C \in \mathcal C_+(G_n)$. 
Conditional on~$\mathcal C_+(F_n)= C$,~$F_n$ is given by an algorithm of~$w_-$-loop erased random walks with wired conditions on~$C$. 

The event~$\Gamma^+_l \cap A_m$ is satisfied if the~$w_-$-loop erased random walk starting from x hits a cycle in~$C_+(F_{B_l^x})$ before leaving~$B_l^x$ and before being rooted to another cycle and if one of the~$w_-$-loop erased random walks starting from points of ~$\partial B_m^x$ reaches~$B_l^x$ before being rooted to a cycle or hitting~$C$. 
Therefore, from Lemma~\ref{rootinghitting} and from the union bound,
$$ \mu_n(A_m \cap \Gamma^+_l) \leq \vert \partial B_{m_0}^x \vert  \delta^{m_0}~$$
which concludes the proof.
 \end{proof}
  
 \begin{lemma}\label{cycle+}
 For every vertex~$x \in V$ of~$G$, we have
 ~$$ \mu(\{x \leftrightarrow C_+(F) \} \cap \{ \vert cc(x) \vert = \infty \}) =0.$$
 \end{lemma}
 
 \begin{proof}
 Let~$x \in V$ be a fixed vertex of~$G$ and let
~$$ A := \{ \vert cc(x) \vert = \infty \} = \bigcap_m A_m$$
 be the event that the connected component of 0 is infinite, where~$A_m$ was defined in Definition~\ref{defprobaconnect+}.
  
  Let~$l \in \mathbb N$. Let~$\epsilon >0$. Let~$m_0$ large enough such that~$\vert \partial B_{m_0}^x \vert  \delta^{m_0} \leq \epsilon$. Let~$m$ be large enough such that for every~$y \in \partial B_m^x$,~$B_{m_0}^y \cap B_l^x = \emptyset$. Let~$n$ be large enough such that for every~$y \in \partial B_m^x$,~$B_{m_0}^y \subset G_n$.
 
 Let~$F_n$ distributed according~$\mu_n$.
 Then, from Lemma~\ref{bordcycle+},
$$ \mu_n(A_m \cap \Gamma^+_l) \leq \vert \partial B_{m_0}^x \vert  \delta^{m_0}  \leq \epsilon.$$
 
Since this inequality holds for every~$n$ large enough and~$A_m \cap \Gamma^+_l$ depends on finitely many edges,
$$\mu(A_m \cap \Gamma^+_l) = \lim_n \mu_n(A_m \cap \Gamma^+_l) \leq \epsilon.$$
 
Since this inequality holds for every~$\epsilon$ for~$m$ large enough, we obtain when~$m \rightarrow \infty$,
$$\mu(A_m \cap \Gamma^+_l) \rightarrow 0.$$
 
Since the sequence of events~$(A_m)_m$ is decreasing,
$$ \mu(A \cap \Gamma^+_l) = \mu  \left(\bigcap_m A_m \cap \Gamma^+_l\right) = \lim_m \mu(A_m \cap \Gamma^+_l) = 0.~$$

Since the sequence of events~$(\Gamma^+_l)_l$ is increasing and
$$\Gamma^+ = \{x \leftrightarrow C_+(F)\} = \bigcup_l \Gamma^+_l,$$
we have 
$$ \mu(A \cap \Gamma^+) = \mu \left(A \cap  \left(\bigcup_l \Gamma^+_l \right)\right) = \mu \left(\bigcup_l (A \cap \Gamma^+_l)\right) = \lim_l \mu(A \cap \Gamma^+_l) = 0.~$$
which is precisely what we wanted to prove.
\end{proof}

Let us emphasize that Lemma~\ref{cycle+} and Lemma~\ref{cycle-} show that for every vertex~$x \in V$ of~$G$, almost surely, if~$x$ is connected to a cycle in the random configuration~$F$, the connected component of~$x$ is finite. Therefore, since~$G$ is countable, we immediately deduce the following theorem.  
 
 \begin{theoreme}\label{cyclefinite}
 Under a measure~$\mu_w$ such that~$w_-$ satisfies Assumption~\ref{assumption2}, every connected component with a cycle is finite.  
 \end{theoreme}

From Proposition~\ref{onecycle}, we know that every finite connected component has a cycle then, if a connected component does not have a cycle, it is necessarily an infinite tree. Therefore, almost surely every connected component is either a finite cycle-rooted tree or an infinite tree.

\section*{Conclusion and open questions}

When a positive weight function on oriented cycles takes values in $[0,1]$ and satisfies an assumption of minoration of weights, it gives rise to a unique infinite volume measure on cycle-rooted spanning forests, which is sampled by an algorithm of loop-erased random walks and which is the thermodynamic limit of finite volume measures, with respect to free or wired boundary conditions. Under this measure, almost surely, all connected components are finite and the edge-to-edge correlations decay is exponential. 

By contrast, when the weight function is constant equal to 0, the model is the uniform spanning tree and the thermodynamic limit in infinite volume of finite volume measures is the free or wired uniform spanning forests measure, depending on boundary conditions. On a large class of graphs (amenable graphs for instance), the infinite volume measure does not depend on the boundary conditions and is sampled by the Wilson algorithm of loop-erased random walks (see \cite{10.1214/aop/1176989121, benjamini_uniform_2001, lyons_probability_2017}). Under this measure, almost surely, every connected component is an infinite tree and the edge-to-edge correlations have long range (for instance, they decay polynomially for~$\mathbb Z^d$).

Considering these two cases as instances of a same model, we thus observe two qualitatively distinct phases, depending on the weight function on cycles. 

For determinantal measures on cycle-rooted spanning forests (see \cite{kenyon_spanning_2011}) associated to a unitary connection, sequences of measures on finite growing subgraphs also converge towards infinite volume measures (see \cite{kenyon_determinantal_2019, sun_toroidal_2016, 10.1214/15-AOP1078, kassel_determinantal_2020}). When the connection satisfies some assumptions, the infinite volume measure does not depend on the boundary conditions (see \cite{KL6, these}). 

These determinantal measures are associated to a weight function on cycles which can take values larger than 1, like in Section~\ref{section6}. Under some assumptions on the connection, the assumption of minoration of cycle weights (Assumption~\ref{assumption2}) is satisfied and therefore, by Theorem~\ref{cyclefinite}, almost surely all connected components are either finite cycle-rooted trees or infinite trees. We also observe two phases (polynomial versus exponential decay of edge-to-edge correlations) depending on the unitary connection (see \cite{these}).

A relevant question is to know if under the assumption of minoration on the weight function, there are infinite trees with a positive probability under the infinite volume measure, in particular in the case where the measure is determinantal and associated to a weight function which is provided by a unitary connection.

\section*{Acknowledgments}

We thank Adrien Kassel for suggesting this topic to us and for guidance throughout its study. We also thank Titus Lupu, B\'eatrice de Tili\`ere, C\'edric Boutillier and Kilian Raschel for helpful conversations. Financial support was partly provided by ANR grant number ANR-18-CE40-0033.

\bibliographystyle{alpha}
\bibliography{bibli}

\end{document}